\documentclass[a4paper,12pt]{article}

\usepackage[top=3cm,bottom=3cm,left=2cm,right=2cm]{geometry}
\usepackage{graphics}
\usepackage{graphicx}

\usepackage{amssymb}
\usepackage{amsmath}
\usepackage{amsthm}
\usepackage{dsfont}
\usepackage{marvosym}
\usepackage{multicol}
\usepackage{wasysym}

\usepackage[english]{babel}
\usepackage{cmap}
\usepackage[utf8]{inputenc}
\usepackage[T1]{fontenc}
\usepackage{hrlatex}
\usepackage{mathtools}

\usepackage{setspace}

\allowdisplaybreaks 
\usepackage{setspace}

\makeatletter
 
  \@addtoreset{equation}{section}
\makeatother

\makeatletter
\def\tbcaption{\def\@captype{table}\caption}
\makeatother

\newtheorem{thm}{Theorem}

\newtheorem{lem}[thm]{Lemma}

\newtheorem{cor}[thm]{Corollary}

\theoremstyle{remark}
\newtheorem{remark}[thm]{Remark}

\newtheorem*{theorem*}{Theorem}

\DeclarePairedDelimiter\ceil{\lceil}{\rceil}
\DeclarePairedDelimiter\floor{\lfloor}{\rfloor}
\DeclarePairedDelimiter\fr{\lbrace}{\rbrace}
\DeclarePairedDelimiterX{\cif}[1]{(}{)}{\delimsize(#1\delimsize)}

\begin{document}

\title{Some Refinements of Formulae Involving Floor and Ceiling Functions}

\author{Luka Podrug\\
    Faculty of Civil Engineering, University of Zagreb\\
    Croatia, Zagreb 10000\\
  \text{luka.podrug@grad.unizg.hr}\\
 \and
Dragutin Svrtan\\
    Department of Mathematics\\
    Faculty of Science, University of Zagreb\\
    Croatia, Zagreb 10000\\
  \text{dragutin.svrtan@gmail.com}}

\date{}
\maketitle

\begin{abstract}
The floor and ceiling functions appear often in mathematics and manipulating sums involving floors and ceilings is a subtle game. Fortunately, the well-known textbook \textit{Concrete Mathematics} provides a nice introduction  with a number of techniques explained and a number of single or double sums treated as exercises. For two such double sums we provide their single-sum analogues. These closed-form identities are given in terms of a dual partition of the multiset (regarded as a partition) of all b-ary digits of a nonnegative integer. We also present the double- and single-sum analogues involving the fractional part function and the shifted fractional part function.
\end{abstract}

\noindent

\section{Introduction}\label{intr}

In this paper we are concerned with computing some sums involving floor and ceiling functions, presenting several new formulae that refine previously known results. For the reader's convenience, we start by stating some basic definitions. 

Let $x$ be any real number. The floor of $x$ is defined as the greatest integer less than or equal to $x$. More precisely, $\floor*{x}=\max\left\lbrace n\:|\:n\leq x,\:n\in\mathbb{Z}\right\rbrace$. The ceiling of $x$ is the least integer greater than or equal to $x$, that is  $\ceil*{x}=\min\left\lbrace n\:|\:n\geq x,\:n\in\mathbb{Z}\right\rbrace$. For every $n\in\mathbb{Z}$ and $x\in\mathbb{R}$ we have $\floor*{n+x}=n+\floor*{x}$ and $\ceil*{n+x}=n+\ceil*{x}$. Finally, the fractional part of $x$ is $\fr*{x}=x-\floor{x}$. We sometimes call $\floor*{x}$ the integer part of $x$, since $x=\floor*{x}+\fr*{x}$ with $0\leq \fr*{x}<1$. For every $n\in\mathbb{Z}$ and $0\leq x<1$ we have $\fr*{n+x}=x$. 

Throughout the paper we use the Iverson notation $\left[P\right]$, which encloses a true-or-false statement $P$ in brackets, whose result is 1 if the statement is true, 0 if the statement is false. For example, $$\left[j=k\right]=\begin{cases}
1, &\text{if $j=k$};\\
0, &\text{if $j\neq k$.}
\end{cases}$$
Our main result is a closed-form expression for the sum $$\sum\limits_{k\geq 1}\floor*{\frac{n+jb^{k-1}}{b^k}},$$ where $b\geq 2$,  $0\leq j<b$ and $n$ is any nonnegative integer. The sum is infinite, but only finitely many of its terms are nonzero, and thus it is a well-defined integer. The new formula is a refinement of the known identity $\sum\limits_{k\geq 1}\sum\limits_{0<j<b}\floor*{\frac{n+jb^{k-1}}{b^k}}=n$ \cite[Ex.\ 44, p.\ 44]{TAOCP}. In the second part of the paper we obtain the closed-form expression for the single-sum ceiling analogue $$\sum\limits_{0\leq k\leq \log_bx}\ceil*{\frac{x+jb^k}{b^{k+1}}},$$ and in Section \ref{sec:frac} for the following double- and single-sum fractional analogues,  $$\sum\limits_{0\leq k\leq \log_bn}\sum\limits_{0<j<b}\fr*{\frac{n+jb^k}{b^{k+1}}},$$ 
 $$\sum\limits_{0\leq k\leq \log_bn}\fr*{\frac{n+jb^k}{b^{k+1}}}.$$

\section{Sums involving floor function}\label{sec:floor}

We start with a problem which appears as an exercise in the {\it Concrete Mathematics} \cite[Ex.\ 22, p.\ 97]{GKP}, but which can be traced back to a much older collection of the ‘‘400 best problems’’ \cite[Problem 4346, p.\ 48]{ODMP}. This problem concerns the evaluation of the sum $S(n)=\sum\limits_{k\geq 1}\floor*{\frac{n}{2^k}+\frac{1}{2}}$  and we state two most common ways of dealing with it. 

In the first approach, for a fixed integer $k$, let $f_k(n)=\floor*{\frac{n}{2^k}+\frac{1}{2}}$. The function $x\to \frac{x}{2^k}+\frac{1}{2}$ is clearly a continuous, monotonically increasing function. If $\frac{n}{2^k}+\frac{1}{2}=m\in\mathbb{N}$ for some $n$, then we have $f_k(n-1)<f_k(n)$. For what value(s) of $n$ does this happen? 
\begin{align*}
 \frac{n+2^{k-1}}{2^k}=m \: \Longrightarrow \: n=2^{k-1}\underbrace{\left(2m-1\right)}_{\text{odd}}.     
\end{align*}
Hence, $f_k(n-1)<f_k(n)$ when $n=b\cdot 2^{k-1}$ with $b$ odd. For given $n$, there is exactly one $b$ with that property. So, for fixed $n$, let $b$ be odd and such that $n=b\cdot 2^{k_n-1}$ for some $k_n\geq 1$. We have
\begin{center}$S(n)=\sum\limits_{k\geq 1}f_k(n)$ and $S(n-1)=\sum\limits_{k\geq 1}f_k(n-1)$,\end{center} 
and $f_k(n)=f_k(n-1)$ for every $k\neq k_n$. For $k=k_n$ we have $f_{k_n}(n)=f_{k_n}(n-1)+1$, so $S(n)=S(n-1)+1$ and thus by induction $S(n)=n$.

The second approach to evaluating $S(n)$ (which appears in the original solution) uses binary expansions. Suppose $n$'s binary expansion is $$n=(c_m c_{m-1}\cdots c_1c_0)_2,$$ i.e., $$n=c_m2^m+c_{m-1}2^{m-1}+\cdots+ c_1 2^1+c_0,$$ where each $c_i$ is either $0$ or $1$ and the leading bit $c_m$ is $1$. 

We see that only digits $c_s$ for $s\geq k-1$ contribute, so
\begin{align*}
\floor*{\frac{n}{2^k}+\frac{1}{2}}&=\floor*{\frac{\sum\limits_{s\geq 0} c_s2^s}{2^k}+\frac{1}{2}}\\
&=\floor*{\sum\limits_{s\geq 0} c_s2^{s-k}+\frac{1}{2}}\\
&=\floor*{\sum\limits_{0\leq s\leq k-1} c_s2^{s-k}+\sum\limits_{s\geq k} c_s2^{s-k}+\frac{1}{2}}\\
&=\sum\limits_{s\geq k} c_s2^{s-k}+\floor*{\sum\limits_{0\leq s\leq k-1} c_s2^{s-k}+\frac{1}{2}}\\
&=\sum\limits_{s\geq k} c_s2^{s-k}+c_{k-1}.
\end{align*}
By summing over $k\geq 1$ we get
\begin{align*}
\sum\limits_{k\geq 1}\floor*{\frac{n}{2^k}+\frac{1}{2}}&=\sum\limits_{k\geq 0}c_k+\sum\limits_{k\geq 1}\sum\limits_{s\geq k}c_s2^{s-k}\\
&=\sum\limits_{k\geq 0}c_k+\sum\limits_{s\geq 1}\sum\limits_{1\leq k\leq s}c_s2^{s-k}\\
&=\sum\limits_{k\geq 0}c_k+\sum\limits_{s\geq 1} c_s\frac{2^s-1}{2-1}\\
&=\sum\limits_{k\geq 0}c_k+\sum\limits_{s\geq 1} 2^sc_s-\sum\limits_{s\geq 1}c_s\\
&=\sum\limits_{s\geq 0}2^sc_s\\
&=n.
\end{align*}

More generally, for every base $b$ we have the following identity \cite[Ex.\ 44, p.\ 44]{TAOCP}
\begin{equation}\label{floorsum}
\sum\limits_{k\geq 1}\sum\limits_{0<j<b}\floor*{\frac{n+jb^{k-1}}{b^k}}=n.
\end{equation}
To prove it, we can use the following replication identity \cite[3.26, p.\ 85]{GKP} $$\floor*{x}+\floor*{x+\frac{1}{m}}+\floor*{x+\frac{2}{m}}+\cdots+\floor*{x+\frac{m-1}{m}}=\floor*{mx}$$ which leads to
\begin{align*}
\sum\limits_{k\geq 1}\sum\limits_{0<j<b}\floor*{\frac{n+jb^{k-1}}{b^k}}&=\sum\limits_{k\geq 1}\sum\limits_{0<j<b}\floor*{\frac{n}{b^k}+\frac{j}{b}}\\
&=\sum\limits_{k\geq 1}\left(\floor*{\frac{n}{b^{k-1}}}-\floor*{\frac{n}{b^k}}\right)\\
&=n.
\end{align*}

To move toward to a new refined result (stated in Theorem \ref{theorem1} below) we need more notation. For every base $b\geq 2$ let $n$'s $b$-ary expansion be
$n=\left(c_m c_{m-1}\cdots c_0\right)_b$, that is $$n=c_mb^m+c_{m-1}b^{m-1}+\cdots+c_1b+c_0,$$ where $c_i\in\left\lbrace 0,1,\dots,b-1\right\rbrace$, $m=\floor*{\log_bn}$ and where the leading $b$-ary digit $c_m$ is nonzero. Consider a multiset of $b$-ary digits $\left\lbrace c_0,c_1,\dots,c_m\right\rbrace$ and let $s_b(n)$ denote the sum of all $b$-ary digits of $n$, that is $$s_b(n)=s_b\left(\sum\limits_{s=0}^{\floor*{\log_bn}}c_s b^s\right)=\sum_{s=0}^{\floor*{\log_bn}} c_s.$$ Let $\lambda=\lambda(n,b)=\left( \lambda_1\geq\lambda_2\geq\cdots\lambda_m\right)$ be the partition obtained by writing the $b$-ary digits $\left\lbrace c_0,c_1,\dots,c_m\right\rbrace$ in descending order and let $\lambda'=\lambda'(n,b)=\left( \lambda'_1\geq\lambda'_2\geq\cdots\lambda'_{b-1}\right)$ be the transpose of the partition $\lambda$. This follows Macdonald's notation \cite{MCD}. Note that $\lambda'_j$ is the number of $b$-ary digits greater than or equal to $j$. For example, let $n=1024$ and $b=3$. We have $1024=(1101221)_3$. Then $s_3(1024)=8$, $\lambda=(2\geq 2\geq 1\geq 1\geq 1\geq 1\geq 0)$ and $\lambda'=(6\geq 2)$.
\begin{remark}\label{remark1} Observe that both $\lambda$ and $\lambda'$ are partitions of the number $s_b(n)$, i.e., $$s_b(n)=\left|\lambda\right|=\sum_{0\leq j\leq m}\lambda_j=\sum_{0<j<b}\lambda'_j.$$
\end{remark}

\begin{thm}\label{theorem1}
For every base $b\geq 2$, fixed $j$, $0\leq j<b$, and for every nonnegative integer $n$ we have the following identity: \begin{equation}\label{floorsinglesum}\sum\limits_{k\geq 1}\floor*{\frac{n+jb^{k-1}}{b^k}}=\frac{n-s_b(n)}{b-1}+\lambda'_{b-j}.\end{equation}\end{thm}
For $j=0$ and $b=p$ prime one has the well-known number-theoretic formula $$\nu_p(n!)=\sum\limits_{k\geq 1}\floor*{\frac{n}{p^k}}=\frac{n-s_p(n)}{p-1},$$ where $\nu_p(m)$ is the multiplicity of $p$ in the factorization of $m$.

First we state and prove an auxiliary result.
\begin{lem} \label{lemma}
For every base $b$, let $n$'s $b$-ary expansion be $n=\left(c_mc_{m-1}\cdots c_0\right)_b$ with $0\leq c_s\leq b-1$ for every $s$ and $0\leq j<b$. Then for every $k$ $$\floor*{\frac{c_k+j }{b}+\frac{c_{k-1}}{b^2}+\cdots+\frac{c_{0}}{b^{k+1}}}=\begin{cases} 0, &\text{if $c_k+j<b$;}\\
1, &\text{if $c_k+j\geq b$.} \end{cases}$$ \end{lem}

\begin{proof}[Proof of Lemma \ref{lemma}] Since
\begin{align*}
\floor*{\frac{c_k+j}{b}+\frac{c_{k-1}}{b^2}+\cdots+\frac{c_{0}}{b^{k+1}}}&\leq\floor*{\frac{2b-2}{b}+\frac{b-1}{b^2}+\cdots+\frac{b-1}{b^{k+1}}}\\
&=\floor*{2-\frac{2}{b}+\frac{1}{b}-\frac{1}{b^2}+\frac{1}{b^2}-\frac{1}{b^3}+\cdots+\frac{1}{b^k}-\frac{1}{b^{k+1}}}\\
&=\floor*{2-\frac{1}{b}-\frac{1}{b^{k+1}}}\\
&<2,
\end{align*}
the integer part $\floor*{\frac{c_k+j}{b}+\frac{c_{k-1}}{b^2}+\cdots+\frac{c_{0}}{b^{k+1}}}$ is either $0$ or $1$.

For $c_k+j\geq b$ we have $\floor*{\frac{c_k+j}{b}+\frac{c_{k-1}}{b^2}+\cdots+\frac{c_{0}}{b^{k+1}}}=1$, and if $c_k+j<b$ we have $c_k+j\leq b-1$, thus
\begin{align*}
0\leq\floor*{\frac{c_k+j}{b}+\frac{c_{k-1}}{b^2}+\cdots+\frac{c_{0}}{b^{k+1}}}&\leq\floor*{\frac{b-1}{b}+\frac{b-1}{b^2}+\cdots+\frac{b-1}{b^{k+1}}}\\
&=\floor*{1-\frac{1}{b}+\frac{1}{b}-\frac{1}{b^{k-1}}}\\
&=\floor*{1-\frac{1}{b^{k+1}}}\\
&=0.
\end{align*}
\end{proof}

\begin{proof} [Proof of Theorem \ref{theorem1}] By substituting $n$ with the $b$-ary expansion of $n$ and by using the Lemma \ref{lemma} we obtain 
\begin{align*}
\sum\limits_{k\geq 1}\floor*{\frac{n+jb^{k-1}}{b^k}}&=\sum\limits_{k\geq 1}\floor*{\frac{n}{b^k}+\frac{j}{b}}\\
&=\sum\limits_{k\geq 1}\floor*{\sum\limits_{s\geq 0}\frac{c_sb^s}{b^k}+\frac{j}{b}}\\
&=\sum\limits_{k\geq 1}\floor*{\sum\limits_{s\geq k}c_sb^{s-k}+\frac{c_{k-1}+j}{b}+\frac{c_{k-2}}{b^2}+\cdots+\frac{c_{0}}{b^k}}\\
&=\sum\limits_{k\geq 1}\sum\limits_{s\geq k}c_sb^{s-k}+\sum\limits_{k\geq 1}\floor*{\frac{c_{k-1}+j}{b}+\frac{c_{k-2}}{b^2}+\cdots+\frac{c_{0}}{b^k}}\\
&=\sum\limits_{k\geq 1}\sum\limits_{s\geq k}c_sb^{s-k}+\sum\limits_{k\geq 1}\left[c_{k-1}+j\geq b\right]\, & \left(\text{by Lemma \ref{lemma}}\right)\\
&=\sum\limits_{s\geq 1}c_s\left(b^0+b^1+\cdots+b^{s-1}\right)+\sum\limits_{s\geq 0}\left[c_{s}\geq b-j\right]\\
&=\sum\limits_{s\geq 1}c_s\cdot \frac{b^s-1}{b-1}+\lambda'_{b-j}\\
&=\frac{1}{b-1}\left(\sum\limits_{s\geq 1}c_sb^s+c_0-c_0-\sum\limits_{s\geq 1}c_s\right)+\lambda'_{b-j}\\
&=\frac{1}{b-1}\left(\sum\limits_{s\geq 0}c_sb^s-\sum\limits_{s\geq 0}c_s\right)+\lambda'_{b-j}\\
&=\frac{n-s_b(n)}{b-1}+\lambda'_{b-j}.
\end{align*}
\end{proof}

The double-sum Formula \eqref{floorsum} then follows immediately:
\begin{cor}For every base $b\geq 2$ and for every nonnegative integer $n$ we have the following identity: $$\sum\limits_{k\geq 1}\sum\limits_{0<j<b}\floor*{\frac{n+jb^{k-1}}{b^{k}}}=n.$$ \end{cor}

\begin{proof} By changing the order of summation the left-hand side is equal to
\begin{align*}
\sum\limits_{k\geq 1}\sum\limits_{0<j<b}\floor*{\frac{n+j b^{k-1}}{b^{k}}}&=\sum\limits_{0<j<b}\sum\limits_{k\geq 1}\floor*{\frac{n+j b^{k-1}}{b^{k}}}\\
&=\sum\limits_{0<j<b}\left(\frac{n-s_b(n)}{b-1}+\lambda_{b-j}'\right) \;&\left(\text{by Theorem \ref{theorem1}}\right)\\
&=n-s_b(n)+\left(\lambda_{b-1}+\cdots+\lambda_{1} \right)\\
&=n.& \left(\text{by Remark \ref{remark1}}\right)
\end{align*}
\end{proof}

\section{Sums involving ceiling function}

Now we turn our attention to the well-known ceiling analogue of the Formula \eqref{floorsum}. A double sum \cite[Ex.\ 39, p.\ 99]{GKP} $$\sum\limits_{0\leq k\leq \log_bx}\sum\limits_{0<j<b}\ceil*{\frac{x+jb^k}{b^{k+1}}}$$ is equal to $(b-1)\left(\floor*{\log_bx}+1\right)+\ceil*{x}-1$ for every real number $x\geq 1$ and every integer $b\geq 2$. 

More generally, we state the following refinement of this formula. Again, let $b\geq 2$ be any base, let $n:=\ceil*{x}$ and let $n=\left(c_m c_{m-1}\cdots c_0\right)_b$ be $n$'s $b$-ary expansion, $\lambda=\lambda(n,b)$ and $\lambda'=\lambda'(n,b)$, as before, partition of the multiset of the $b$-ary digits of $n$ and its transpose. Note that $m=\floor*{\log_bn}$ is the position of the leading digit in $n$'s $b$-ary expansion. We also define $\nu_b(n):=\text{max}\left\lbrace k \:|\:b^k\backslash n\right\rbrace$. We shall evaluate the sum for every fixed $j$, $0<j<b$. 

\begin{thm}\label{theorem2} For every base $b\geq 2$, fixed $j$, $0<j<b$, and for every real $x\geq 1$ we have the following identity:\begin{equation} \label{ceilsinglesum} 
\sum\limits_{0\leq k\leq \log_bx}\ceil*{\frac{x+jb^k}{b^{k+1}}}=\frac{n-s_b\left(n\right)}{b-1}+m+\lambda'_{b-j}+\left[c_{\nu_b(n)}\neq b-j\right],\end{equation} where $n=\ceil*{x}$ and $m=\floor{\log_bn}$ is the position of the leading digit in $n$'s $b$-ary expansion. \end{thm}
\begin{proof}

Let $b\geq 2$, $0<j<b$ and $0\leq k\leq \log_bx$ be integers and $f(x)=\frac{x+jb^k}{b^{k+1}}$. Then $f$ is a continuous, monotonically increasing function. Suppose $f(x)=z\in\mathbb{Z}$. Then we have $x=zb^{k+1}-jb^k\in\mathbb{Z}$ and we can conclude $\ceil*{f(x)}=\ceil*{f\left(\ceil*{x}\right)}$ since $f$ satisfies the necessary conditions \cite[(3.10), p. 71]{GKP}. Then  
$$\sum\limits_{0\leq k\leq \log_bx}\ceil*{\frac{x+jb^k}{b^{k+1}}}=\sum\limits_{0\leq k\leq \log_bx}\ceil*{\frac{n+jb^k}{b^{k+1}}}.$$

Suppose $n$'s $b$-ary expansion is $$n=(c_m c_{m-1}\cdots c_1 c_0)_b,$$ i.e., $$n=c_mb^m+c_{m-1}b^{m-1}+\cdots+ c_1b^1+c_0,$$ where each $c_s\in\left\lbrace 0,1,\dots ,b-1\right\rbrace$ and the leading digit $c_m$ is nonzero. 

\begin{align*}
\sum\limits_{0\leq k\leq \log_bx}\ceil*{\frac{x+jb^k}{b^{k+1}}}&=\sum\limits_{0\leq k\leq \log_bx}\ceil*{\frac{n}{b^{k+1}}+\frac{j}{b}}\\
&=\sum\limits_{0\leq k\leq \log_bx}\ceil*{\sum\limits_{s\geq 0}\frac{c_sb^s}{b^{k+1}}+\frac{j}{b}}\\
&=\sum\limits_{0\leq k\leq \log_bx}\ceil*{\sum\limits_{s\geq k+1}\frac{c_sb^s}{b^{k+1}}+\sum\limits_{0\leq s\leq k}\frac{c_sb^s}{b^{k+1}}+\frac{j}{b}}\\
&=\sum\limits_{0\leq k\leq \log_bx}\sum\limits_{s\geq k+1}c_sb^{s-k-1}+\sum\limits_{0\leq k\leq \log_bx}\ceil*{\frac{c_k+j}{b}+\frac{c_{k-1}}{b^2}+\cdots+\frac{c_0}{b^{k+1}}}.
\end{align*}
By changing the order of summation, the first (double) sum is equal to
\begin{align*}
\sum\limits_{0\leq k\leq \log_bx}\sum\limits_{s\geq k+1}c_sb^{s-k-1}&=\sum\limits_{1 \leq s \leq\log_bx+1}c_sb^{s-1}\sum\limits_{0\leq k\leq s-1}b^{-k}\\
&=\sum\limits_{1 \leq s \leq\log_bx+1}c_sb^{s-1}\cdot\frac{1-b^{-s}}{1-b^{-1}}\\
&=\sum\limits_{1 \leq s \leq\log_bx+1}c_s\cdot\frac{b^{s-1}-b^{-1}}{1-b^{-1}}\\
&=\sum\limits_{1 \leq s \leq\log_bx+1}c_s\cdot\frac{b^{s}-1}{b-1}\\
&=\frac{1}{b-1}\left(\sum\limits_{1 \leq s \leq\log_bx+1}c_s b^s-\sum\limits_{1 \leq s \leq\log_bx+1}c_s\right)\\
&=\frac{n-s_b\left(n\right)}{b-1}.
\end{align*}
The second (single) sum is bounded by
\begin{align*}
\ceil*{\frac{c_k+j}{b}+\frac{c_{k-1}}{b^2}+\cdots+\frac{c_0}{b^{k+1}}}&\leq\ceil*{\frac{2b-2}{b}+\frac{b-1}{b^2}+\cdots+\frac{b-1}{b^{k+1}}}\\
&=\ceil*{2-\frac{2}{b}+\frac{1}{b}-\frac{1}{b^2}+\frac{1}{b^2}-\frac{1}{b^3}+\cdots+\frac{1}{b^k}-\frac{1}{b^{k+1}}}\\
&=\ceil*{2-\frac{1}{b}-\frac{1}{b^{k+1}}}\\
&\leq 2.
\end{align*}

Hence, $\ceil*{\frac{c_k+j}{b}+\frac{c_{k-1}}{b^2}+\cdots+\frac{c_0}{b^{k+1}}}$ is either $1$ or $2$.

For $c_k+j>b$ we have $\ceil*{\frac{c_k+j}{b}+\frac{c_{k-1}}{b^2}+\cdots+\frac{c_0}{b^{k+1}}}=2$. For $c_k+j<b$ we have $c_k+j\leq b-1$ and
\begin{align*}
1\leq\ceil*{\frac{c_k+j}{b}+\frac{c_{k-1}}{b^2}+\cdots+\frac{c_0}{b^{k+1}}}&\leq\ceil*{\frac{b-1}{b}+\frac{b-1}{b^2}+\cdots+\frac{b-1}{b^{k+1}}}\\
&=\ceil*{1-\frac{1}{b}+\frac{1}{b}-\frac{1}{b^2}+\frac{1}{b^2}-\frac{1}{b^3}+\cdots+\frac{1}{b^k}-\frac{1}{b^{k+1}}}\\
&=\ceil*{1-\frac{1}{b^{k-1}}}\\
&=1.
\end{align*}
For $c_k+j=b$ and $c_{k-1}=c_{k-2}=\cdots =c_0=0$ we have that $b^k$ exactly divides $n$, so $\nu_b(n)=k$  and $$\ceil*{\frac{c_k+j}{b}+\frac{c_{k-1}}{b^2}+\cdots+\frac{c_0}{b^{k+1}}}=1.$$ For at least one $c_s>0$ with $0\leq s\leq k-1$ we have  $\ceil*{\frac{c_k+j}{b}+\frac{c_{k-1}}{b^2}+\cdots+\frac{c_0}{b^{k+1}}}=2$. So:
\begin{align*}
\sum\limits_{0\leq k\leq \log_bx}\ceil*{\frac{c_k+j}{b}+\frac{c_{k-1}}{b^2}+\cdots+\frac{c_0}{b^{k+1}}}&=m+1+\lambda'_{b-j}-\left[c_{\nu_b(n)}=b-j\right]\\
&=m+\lambda'_{b-j}+\left[c_{\nu_b(n)}\neq b-j\right].
\end{align*}
Finally, we have $$\sum\limits_{0\leq k\leq \log_bx}\ceil*{\frac{x+jb^k}{b^{k+1}}}=\frac{n-s_b\left(n\right)}{b-1}+m+\lambda'_{b-j}+\left[c_{\nu_b(n)}\neq b-j\right].$$
\end{proof}

\begin{cor}For every base $b\geq 2$ and for every real $x\geq 1$ we have the following identity: $$\sum\limits_{0\leq k\leq \log_bx}\sum\limits_{0<j<b}\ceil*{\frac{x+jb^k}{b^{k+1}}}=(b-1)\left(m+1\right)+n-1,$$
where $n=\ceil*{x}$ and $m=\floor{\log_bn}$ is the position of the leading digit in $n$'s $b$-ary expansion.\end{cor}
\begin{proof} 
\begin{align*}
\sum\limits_{0\leq k\leq \log_bx}\hspace{0.5mm}\sum\limits_{0<j<b}\ceil*{\frac{x+jb^k}{b^{k+1}}}&=\sum\limits_{0<j<b}\hspace{0.5mm}\sum\limits_{0\leq k\leq \log_bx}\ceil*{\frac{x+jb^k}{b^{k+1}}}\\
&=\sum\limits_{0<j<b}\left(\frac{n-s_b\left(n\right)}{b-1}+m+1+\lambda'_{b-j}-\left[c_{\nu_b(n)}=b-j\right]\right)\\
&=n-s_b\left(n\right)+(b-1)\left(m+1\right)+\sum\limits_{0<j<b}\left(\lambda'_{b-j}-\left[c_{\nu_b(n)}=b-j\right]\right).
\end{align*}
It remains to prove that $$\sum\limits_{0<j<b}\left(\lambda'_{b-j}-\left[c_{\nu_b(n)}=b-j\right]\right)=s_b\left(n\right)-1.$$
By Remark \ref{remark1}, $\sum\limits_{0<j<b}\lambda'_{b-j}=s_b\left(n\right)$. We can also see that $$\sum\limits_{0<j<b}\left[c_{\nu_b(n)}=b-j\right]=1.$$ This is true because there is exactly one $k$ that $\nu_b(n)=k$. Since $c_{\nu_b(n)}>0$ and $j$ is ranging from $1$ to $b-1$, the required digit that makes expression $\left[c_{\nu_b(n)}=b-j\right]$ equal to $1$ must also occur and also exactly once. 

Finally, we have
\begin{align*}
\sum\limits_{0\leq k\leq \log_bx}\hspace{0.5mm}\sum\limits_{0<j<b}\ceil*{\frac{x+jb^k}{b^{k+1}}}&=n-s_b\left(n\right)+(b-1)\left(m+1\right)+\sum\limits_{0<j<b}\left(\lambda'_{b-j}-\left[c_{\nu_b(n)}=b-j\right]\right)\\
&=n-s_b\left(n\right)+(b-1)\left(m+1\right)+s_b\left(n\right)-1\\
&=(b-1)\left(m+1\right)+n-1.\end{align*} \end{proof}

\section{Sums involving fractional part function}\label{sec:frac}
In the following theorem we obtain the fractional analogue of the floor and ceiling Formulae \eqref{floorsinglesum} and \eqref{ceilsinglesum}.
\begin{thm}\label{theorem3}
For every base $b\geq 2$, fixed $j$, $0<j<b$, and for every nonnegative integer $n$ we have the following identity:
$$\sum\limits_{0\leq k\leq \log_bn}\fr*{\frac{n+jb^k}{b^{k+1}}}=\dfrac{1}{b-1}\left(s_b(n)-\dfrac{n}{b^{m+1}}\right)+(m+1)\frac{j}{b}-\lambda_{b-j}',$$ where $m=\floor{\log_bn}$ is the position of the leading digit in $n$'s $b$-ary expansion. 
\end{thm}
\begin{proof}
Suppose $n$'s $b$-ary expansion is $$n=(c_m c_{m-1}\cdots c_1 c_0)_b,$$ that is $$n=c_mb^m+c_{m-1}b^{m-1}+\cdots+c_1b^1+c_0,$$ where each $c_s\in\left\lbrace 0,1,\dots ,b-1\right\rbrace$ and the leading digit $c_m$ is nonzero.

\begin{align*}
\sum\limits_{0\leq k\leq \log_bn}\fr*{\frac{n+jb^k}{b^{k+1}}}
&=\sum\limits_{0\leq k\leq \log_bn}\fr*{\sum\limits_{s\geq 0}\frac{c_sb^s}{b^{k+1}}+\frac{j}{b}}\\
&=\sum\limits_{0\leq k\leq \log_bn}\fr*{\sum\limits_{s\geq k+1}\frac{c_sb^s}{b^{k+1}}+\sum\limits_{0\leq s\leq k}\frac{c_sb^s}{b^{k+1}}+\frac{j}{b}}\\
&=\sum\limits_{0\leq k\leq \log_bn}\fr*{\sum\limits_{0\leq s\leq k}\frac{c_sb^s}{b^{k+1}}+\frac{j}{b}}.
\end{align*}

The inner sum is bounded by 
\begin{align*}
\frac{c_k+j}{b}+\frac{c_{k-1}}{b^2}+\cdots+\frac{c_0}{b^{k+1}}&\leq\frac{2b-2}{b}+\frac{b-1}{b^2}+\cdots+\frac{b-1}{b^{k+1}}\\
&=2-\frac{1}{b}-\frac{1}{b^{k+1}}\\
&<2.
\end{align*}

So, $\fr*{\sum\limits_{0\leq s\leq k}\frac{c_sb^s}{b^{k+1}}+\frac{j}{b}}$ is either $\sum\limits_{0\leq s\leq k}\frac{c_sb^s}{b^{k+1}}+\frac{j}{b}$ or $\sum\limits_{0\leq s\leq k}\frac{c_sb^s}{b^{k+1}}+\frac{j}{b}-1$.
For $c_k+j\geq b$ we have $\fr*{\sum\limits_{0\leq s\leq k}\frac{c_sb^s}{b^{k+1}}+\frac{j}{b}}\geq 1$ and $\fr*{\sum\limits_{0\leq s\leq k}\frac{c_sb^s}{b^{k+1}}+\frac{j}{b}}=\sum\limits_{0\leq s\leq k}\frac{c_sb^s}{b^{k+1}}+\frac{j}{b}-1$. 
For $c_k+j<b$ we have $c_k+j\leq b-1$ and
\begin{align*}
\fr*{\frac{c_k+j}{b}+\frac{c_{k-1}}{b^2}+\cdots+\frac{c_0}{b^{k+1}}}&\leq\fr*{\frac{b-1}{b}+\frac{b-1}{b^2}+\cdots+\frac{b-1}{b^{k+1}}}\\
&=1-\frac{1}{b^{k-1}}\\
&<1.
\end{align*}
We conclude that $\fr*{\sum\limits_{0\leq s\leq k}\frac{c_sb^s}{b^{k+1}}+\frac{j}{b}}=\sum\limits_{0\leq s\leq k}\frac{c_sb^s}{b^{k+1}}+\frac{j}{b}$.

\begin{align*}
\sum\limits_{0\leq k\leq \log_bn}\fr*{\frac{n+jb^k}{b^{k+1}}}&=\sum\limits_{0\leq k\leq \log_bn}\left(\sum\limits_{0\leq s\leq k}\frac{c_sb^s}{b^{k+1}}+\frac{j}{b}-\left[c_k\geq b-j\right]\right)\\
&=\sum\limits_{0\leq s\leq \log_bn}c_s\sum\limits_{s\leq k\leq \log_bn}b^{s-k-1}+(\floor*{\log_bn}+1)\frac{j}{b}-\lambda_{b-j}'\\
&=\sum\limits_{0\leq s\leq \log_bn}c_s\cdot\dfrac{b^{-1}-b^{s-\floor{\log_bn}-2}}{1-\frac{1}{b}}+(m+1)\frac{j}{b}-\lambda_{b-j}'\\
&=\sum\limits_{0\leq s\leq \log_bn}c_s\cdot\dfrac{1-b^{s-m-1}}{b-1}+(m+1)\frac{j}{b}-\lambda_{b-j}'\\
&=\dfrac{\sum\limits_{0\leq s\leq \log_bn}c_s}{b-1}-\dfrac{\sum\limits_{0\leq s\leq \log_bn}c_sb^{s}}{b^{m+1}(b-1)}+(m+1)\frac{j}{b}-\lambda_{b-j}'\\
&=\dfrac{1}{b-1}\left(s_b(n)-\dfrac{n}{b^{m+1}}\right)+(m+1)\frac{j}{b}-\lambda_{b-j}'.
\end{align*}
\end{proof}

Finally, by simple summation of the formula in Theorem \ref{theorem3} over $0<j<b$, the closed-form formula for the double-sum fractional analogue reads as follows:

\begin{cor}\label{fraccor} For every base $b\geq 2$ and for every nonnegative integer $n$ we have the following identity: $$\sum\limits_{0\leq k\leq \log_bn}\hspace{0.5mm}\sum\limits_{0<j<b}\fr*{\frac{n+jb^k}{b^{k+1}}}=(m+1)\frac{b-1}{2}-\dfrac{n}{b^{m+1}},$$ where $m=\floor{\log_bn}$ is the position of the leading digit in $n$'s $b$-ary expansion. \end{cor}
\begin{proof} By changing the order of summation the left-hand side is equal to
\begin{align*}
\sum\limits_{0<j<b}\sum\limits_{0\leq k\leq \log_bn}\fr*{\frac{n+jb^k}{b^{k+1}}}&=\sum\limits_{0<j<b}\left(\dfrac{1}{b-1}\left(s_b(n)-\dfrac{n}{b^{m+1}}\right)+(m+1)\frac{j}{b}-\lambda_{b-j}'\right)\\
&=s_b(n)-\dfrac{n}{b^{m+1}}+(m+1)\frac{b(b-1)}{2b}-s_b(n)\\
&=(m+1)\frac{b-1}{2}-\dfrac{n}{b^{m+1}}. \end{align*}
\end{proof}

Yet another among ‘‘400 best problems’’ \cite[Problem 4375, p.\ 50]{ODMP} concerns the summation of the following function: $$\cif{x}=\begin{cases}
x-\floor*{x}-\frac{1}{2}\left(=\fr*{x}-\frac{1}{2}\right), &\text{if $x$ is not an integer;}\\
0,&\text{if $x$ is an integer.}
\end{cases}$$  This function is basic for the definition of the well-known Dedekind sum \cite[Def.\ 5.1, p.\ 92]{AST}. Here we can easily obtain the closed-form analogues of single and double sums for this function.

\begin{remark}\label{remarkint} For $0\leq k\leq \log_bn$ and fixed $j$, $0< j<b$, only one among the fractions  $$\frac{n+jb^k}{b^{k+1}}$$ is an integer.\end{remark}
\begin{proof} If $\nu_b(n)=k$ then $c_k\neq 0$ and $n=\sum\limits_{s=k}^m c_s b^s,$ where $m=\floor*{\log_bn}$. Then  $$\frac{n+jb^k}{b^{k+1}}=\frac{c_k+j}{b}+c_{k+1}+c_{k+2}b+\cdots+c_mb^{m-k-1}.$$ Then $c_k=b-j$ gives the only integer value of the fraction above.\end{proof}

\begin{thm}\label{theorem4}
For every base $b\geq 2$, fixed $j$, $0<j<b$, and for every nonnegative integer $n$ we have the following identity:
$$\sum\limits_{0\leq k\leq \log_bn}\cif*{\frac{n+jb^k}{b^{k+1}}}=\dfrac{1}{b-1}\left(s_b(n)-\dfrac{n}{b^{m+1}}\right)+(m+1)\left(\frac{j}{b}-\frac{1}{2}\right)-\lambda_{b-j}'+\frac{1}{2}\left[c_{\nu_b(n)}=b-j\right],$$ where $m=\floor{\log_bn}$ is the position of the leading digit in $n$'s $b$-ary expansion. 
\end{thm}
\begin{proof}
Suppose $n$'s $b$-ary expansion $n=c_mb^m+c_{m-1}b^{m-1}+\cdots+ c_1b^1+c_0$ with $c_m$ is nonzero.
\begin{align*}
&\sum\limits_{0\leq k\leq \log_bn}\cif*{\frac{n+jb^k}{b^{k+1}}}\\
=&\sum\limits_{0\leq k\leq \log_bn}\fr*{\frac{n+jb^k}{b^{k+1}}}-\sum\limits_{0\leq k\leq \log_bn}\frac{1}{2}+\frac{1}{2}\left[c_{\nu_b(n)}=b-j\right]& \left(\textup{by Remark \ref{remarkint}}\right)\\
=&\dfrac{1}{b-1}\left(s_b(n)-\dfrac{n}{b^{m+1}}\right)+(m+1)\frac{j}{b}-\lambda_{b-j}'-\dfrac{m+1}{2}+\frac{1}{2}\left[c_{\nu_b(n)}=b-j\right]\\
=&\dfrac{1}{b-1}\left(s_b(n)-\dfrac{n}{b^{m+1}}\right)+(m+1)\left(\frac{j}{b}-\frac{1}{2}\right)-\lambda_{b-j}'+\frac{1}{2}\left[c_{\nu_b(n)}=b-j\right].
\end{align*}
\end{proof}

\begin{cor} For every base $b\geq 2$ and for every nonnegative integer $n$ we have the following identity: $$\sum\limits_{0\leq k\leq \log_bn}\hspace{0.5mm}\sum\limits_{0<j<b}\cif*{\frac{n+jb^k}{b^{k+1}}}=\dfrac{1}{2}-\dfrac{n}{b^{m+1}},$$ where $m=\floor{\log_bn}$ is the position of the leading digit in $n$'s $b$-ary expansion. \end{cor}
\begin{proof} Note that, by Remark \ref{remarkint}, the sum $\sum\limits_{0<j<b}\frac{1}{2}\left[c_{\nu_b(n)}=b-j\right]$ is equal to $\frac{1}{2}$.
\begin{align*}
&\sum\limits_{0\leq k\leq \log_bn}\sum\limits_{0<j<b}\cif*{\frac{n+jb^k}{b^{k+1}}}\\
=&\sum\limits_{0<j<b}\left(\dfrac{1}{b-1}\left(s_b(n)-\dfrac{n}{b^{m+1}}\right)+(m+1)\left(\frac{j}{b}-\frac{1}{2}\right)-\lambda_{b-j}'+\frac{1}{2}\left[c_{\nu_b(n)}=b-j\right]\right)\\
=&s_b(n)-\dfrac{n}{b^{m+1}}+(m+1)\left(\frac{b(b-1)}{2b}-\frac{b-1}{2}\right)-s_b(n)+\sum\limits_{0<j<b}\frac{1}{2}\left[c_{\nu_b(n)}=b-j\right]\\
=&\frac{1}{2}-\dfrac{n}{b^{m+1}}. 
\end{align*}
\end{proof}

We hope that our summation techniques may be applied further in number theory, approximation theory or elsewhere.


\begin{thebibliography}{[99]}
\bibitem{GKP}
R.\ L.\ Graham, D.\ E.\ Knuth, and O.\ Patashnik, \textit{Concrete Mathematics}, Addison-Wesley, 1994.
\bibitem{TAOCP}
D.\ E. Knuth, \textit{The Art of Computer Programming}, Addison-Wesley, 1997.
\bibitem{TNON}
D.\ E. Knuth, Two notes on notation, \textit{Amer.\ Math.\ Monthly} \textbf{99} (1992), 403--422.
\bibitem{ODMP}
H.\ Eves and E.\ P.\ Starke, \textit{The Otto Dunkel Memorial Problem Book}, Supplement to the \textit{Amer.\ Math.\ Monthly} \textbf{64}, 1957.
\bibitem{AST}
F.\ Hirzebruch and D.\ Zagier, \textit{The Atiyah-Singer Theorem and Elementary Number Theory}, Publish or Perish Inc., 1974.
\bibitem{MCD}
I.\ G.\ Macdonald, \textit{Symmetric Functions and Hall Polynomials},  Oxford University Press Inc., 1995.
\end{thebibliography}
\end{document}